\documentclass[12pt,a4paper]{article}
\usepackage{amssymb}
\usepackage{amsmath,amsfonts,amsthm}
\usepackage{amstext,latexsym,bm}
\usepackage{graphicx}
\usepackage{amsmath}

\usepackage{amsfonts}
\usepackage{amsthm,amscd}
\usepackage{graphicx}
\usepackage{amsmath}
\usepackage{amssymb}
\usepackage{amstext}

\newtheorem{theorem}{Theorem}
\newtheorem{corollary}[theorem]{Corollary}

\newtheorem{example}[theorem]{Example}
\newtheorem{lemma}[theorem]{Lemma}

\newtheorem{proposition}[theorem]{Proposition}

\numberwithin{equation}{section}
\numberwithin{theorem}{section}

\begin{document}

\title{Dynamical hypothesis tests and Decision Theory for Gibbs distributions}
\author{M. Denker, A. O.  Lopes and S. R. C.  Lopes}
\maketitle

\begin{abstract} 
We consider the problem of testing for two Gibbs probabilities $\mu_0$ and $\mu_1$ defined for a dynamical system $(\Omega,T)$. 
Due to the fact that in general full orbits  are not observable or computable, one needs to restrict to subclasses of tests defined by a finite time series  $h(x_0), h(x_1)=h(T(x_0)),..., h(x_n)=h(T^n(x_0))$, $x_0\in \Omega$, $n\ge 0$, where $h:\Omega\to\mathbb R$ denotes a suitable measurable function.
We determine in each class the Neyman-Pearson  tests,  the minimax tests, and the Bayes solutions, and show  the asymptotic  decay of their risk functions, as $n\to\infty$.   In the case of $\Omega$ being a symbolic space, for each $n\in \mathbb{N}$, these optimal tests rely on the information of the measures  for cylinder sets of size $n$.
  \end{abstract}

\section{Introduction}\label{sec:1}

We consider a compact metric space
$\Omega $ with Borel $\sigma$-algebra and the dynamical action of an open and expanding transformation $T$ on  $\Omega$ which is topologically mixing.

Given a H\"older potential, i.e. a H\"older continuous function $A:\Omega \to \mathbb{R}$,
the transfer operator $\mathcal L_A$ associated to $A$ is the one acting on continuous functions $g\in C(\Omega)$ such that
\begin{equation}\label{eq:1.1}  
[\mathcal{L}_{A} g](\omega)=\sum_{\{y\in \Omega\,| T(y)=\omega \}} g(y)\, e^{ A(y)}. \end{equation}
Without loss of generality we may assume  that  for all $\omega\in \Omega$  the Jacobian
\begin{equation}\label{eq:1.2}
J= e^{A}
\end{equation}
 satisfies $J(\omega)>0$ and 
$\sum_{\{y\,| T(y)=\omega\}} J (y) =1$, for every $\omega\in \Omega$.
It is well known that in this case the eigenmeasure $m$ for the eigenvalue $1$ of the dual operator $\mathcal L_A^*$ is $T$-invariant  and is called a Gibbs measure.\footnote{Note that we abuse the terminology for a  Jacobian here since taking $g=\mathbb I_C$ as the indicator function of a set $C$ on which $T$ acts invertible, then $\int_C e^{-A} dm=\int \mathcal L_A \mathbb I_C e^{-A} dm=  m(TC)$. Hence the Jacobian of $T$ is $J^{-1}$.}  Such Gibbs measures have finite Markov partitions $\gamma=\{\Gamma_1,...,\Gamma_d\}$ for some $d \ge 2$.

We make the assumption throughout the paper that we are given two distinct Gibbs measures $\mu_0$ and $\mu_1$ on $\Omega$ which share a common Markov partition $\gamma=\{\Gamma_1,...,\Gamma_d\}$ and that the available information on the orbits of points $\omega$ in $\Omega$ is given by the variables $X_n(\omega)= k\in \{1,...,d\}$ if and only if $T^n(\omega)\in \Gamma_k$. In fact, this is not an essential restriction since such partitions can be obtained for all pairs of Gibbs measures. Their Jacobians will be denoted by $J_0$ and $J_1$, respectively, and we assume that both are strictly positive on their support.

Examples of open, expanding maps include hyperbolic rational functions, certain maps of the interval, expanding differentiable maps on compact manifolds. The results of this paper also hold for invertible maps which admit Markov partitions like Axiom A diffeomorphisms, because we may restrict them to the forward orbit of points.

It follows from the assumption that we can and will restrict to the case when $\Omega=\{1,...,d\}^{\mathbb Z_+}$  for some $d\ge 2$, since Markov partitions  create almost surely one-to-one  maps between the spaces. $\Omega$ is equipped with its Borel $\sigma$-algebra $\mathcal F$. In this setup the two measures $\mu_0$ and $\mu_1$ may be supported on different subspaces of finite type, but both are assumed to be subsets of $\Omega$.	

  We shall be using standard statistical terminology in the sequel as it is also explained in Section \ref{sec:6}, the appendix. Notations, definition and facts of statistical nature used in this note are explained and stated there for the readers convenience.

 The goal of the present note is to decide on $\mu_0$ or $\mu_1$ based on observed data.
Loosely speaking, given a finite sample one has to decide between the  hypothesis $H_0\equiv \mu_0$ and the alternative  $H_1\equiv \mu_1$.
The false alarm or type 1 error happens in case
one announces $H_1$ when, in fact, $H_0$ is true  (that is, the sample was originated by $\mu_0$). The value $0\leq\alpha\leq 1$ denotes the probability of a false alarm, which is called the test size or the significance level of the test. More formally,
$ \alpha:= (\text{Prob. Decide}\, H_1\, | H_0 \, \text{is true})$.

The probability $\beta:=(\text{Prob. Decide}\, H_1 \,\,| \,\, H_1 \,\text{is true})$ is called the power of the test. The value $1-\beta$ is called the probability of type 2 error.  When designing a test one would like to minimize type 1 and  2 errors under some constraints.

 Formally, we consider the  {\it statistical experiment} $\mathcal E:=(\Omega, \mathcal F, \mathcal P)$, where $\mathcal P=\{\mu_0,\mu_1\}$. The objective is to make a decision about the {\it true} probability in $\mathcal P$ once a point in $\Omega$ is observed.

To do this we consider a the  {\it test problem}  which is  specified by the subset $\mathcal H_0=\{\mu_0,\mu_1\}\subset \mathcal P$, the hypothesis  $H_0\equiv\mu_0$ versus $H_1\equiv\mu_1$, the decision space $D=\{0,1\}$ and a loss function $L$ to be set later (see Lemma \ref{lem:2.1}  or the appendix.
A  test can be seen as  a function
$$ \varphi: \Omega\to [0,1]$$
defined as
$$ \varphi(\omega)= \delta(\omega,\{1\})$$
where $\delta\in \Delta$ is a decision function.  

Since a point in the space $\Omega$ is in general not observable one needs to restrict to finite time series. 
Therefore we consider a dynamical setting where test problems $ \mathcal E_n$ are defined for each $n \in \mathbb{N}$. We determine the best tests under Neyman-Pearson, minimax and Bayes distribution constraints and analyze the asymptotic behavior of their error properties, when $n \to \infty$.

We denote by $S$ a set with $\mu_i(S)=i$ ($i=0,1$), which exists since two  distinct Gibbs measures are orthogonal. We denote by $E_m(g)= \int g dm$ the expectation of $g$ with respect to the probability $m$. The first observation is well known  see \cite{Fer}, page 201.

 The Neyman-Pearson Lemma characterizes those tests which have maximal power subject to keeping a given significance level 
$\alpha$. These are called Neyman-Pearson tests.

\begin{theorem}\label{theo:1.1}
The test 
$$\phi^*(\omega)=\begin{cases} 1,\qquad&\mbox{if}\ \omega\in S\\
0, & \mbox{if}\ \omega\not\in S
\end{cases}$$
is a Neyman-Pearson test at level $\alpha=0$ and is as well the minimax test and the Bayes solution for any risk function $\phi\mapsto \mathcal R_{\pi}(\phi)=\pi_0 E_{\mu_0}(\phi)+ \pi_1 E_{\mu_1}(1-\phi)$, $\pi=(\pi_0,\pi_1)$ ($\phi$ any randomized test) where $\pi$ is  the prior distribution on $\{0,1\}$.
\end{theorem}

All other Neyman-Pearson tests for this problem are inferior, so that full information on the orbit requires as well the knowledge of distinct supports of $\mu_0$ and $\mu_1$. So the problem arises to find a good computable test. This can be done using finite time series 
$X_0 , X_1=X_0\circ T,...,X_n=X_0\circ T^n$ ($n\in \mathbb N$) where  $X_0$ is the projection $\Omega\to \{1,...,d\}$ onto the first coordinate. 

We denote by $\mathcal T_n$ the collection of all tests which are measurable with respect to $X_0,...,X_n$. This set can be described by the set of all tests for the test problem 
\begin{equation}\label{eq:1.4}
\mathcal E_n=(\{1,...,d\}^{n+1}, \mathcal P_n=\{\mu_i^n|\ i=0,1\}, H_0^n=\{\mu_0^n\})
\end{equation}
where $\mu_i^n$ ($i=0,1$) denotes the marginal distribution of $\mu_i$ on cylinder sets $c$ of length $n+1$ which are defined as $c=[c_0,...,c_n]=\{\omega\in \Omega |\  \omega_k=c_k\, (0\le k\le n)\}$ ($1\le c_i\le d$  for $0\le i \le n$).

\begin{example}\label{ex:1.2}  In order to illustrate the foregoing  setup, consider the unit interval $\Omega=[0,1]$ together with the map $T(x)=10\cdot x\ \mod 1$. Let $\mu_0$ denote the Lebesgue measure restricted on $\Omega$  and $\mu_1$ the invariant measure associated to a potential $J:\Omega\to \mathbb R_+$ with $\sum_{T(y)=x} J(y)=1$, for all $x\in \Omega$. The Markov partition is just $\gamma=\{ [\frac{j}{10},\frac{j+1}{10}]\ ;0\le j\le 9\}$. More precisely, the potential only needs to be H\"older continuous  with respect to the sequence space metric in $\{0,1,...,9\}^{\mathbb Z_+}$. The test problem then reads as follows: Given an observation $x\in [0,1]$ by its decimal expansion $0.x_0\,x_1\,...\,x_n$ up to the $n+1$st digit, test whether $x$ is more likely to be a generic point for $\mu_0$ or $\mu_1$.
\end{example}

This type of problem was recently studied in \cite{LLV} and \cite{HLL} using Birkhoff averages of the Jacobians to determine the classes of tests. Here we determine the Neyman-Pearson tests for the test problem $\mathcal E_n$ thus deriving the most powerful tests in the class $\mathcal T_n$. We also study the asymptotic behavior of these tests using large deviation theory and determine the minimax tests and Bayes solutions for the test problem $\mathcal E_n$ and show that these tests converge to the minimax test (Bayes solution) for the test problem $\mathcal E$ with exponentially fast decaying risk functions.

Comparing the setting of the present paper with the one in \cite{HLL}, we mention that in \cite{HLL} (which likewise considers hypothesis  tests) it also used LDP properties and a relation with the topological pressure. However,  there  the arguments are concerned just to rejected areas
taking into account  a loss function related to Jacobians, more precisely, $\log J_0 - \log J_1.$
A similar expression like 
 $f_i'(t)= \int (\log J_{i'}-\log J_i) dm_{i,t}$ in Theorem  \ref{theo:3.1} was obtained. One of the main differences is that here we introduced the test $\phi_{n,\alpha}^*$ in Lemma \ref{lem:2.1}, which takes into account the measure of cylinders. This is a different point of view, using a more basic information,  and therefore, much more suitable for applications. Theorem  \ref{theo:3.1} makes the connection of these two points of view. 

The paper \cite{LLV} has a quite different goal. It does not consider hypothesis tests or results on decision theory like here. \cite{LLV}   takes into account the Bayesian point of view, and considers a large class of loss functions (including some non additive expressions which  were not our objective here). The prior probability on the set of parameters $\Theta$ (which does not have to be finite) in \cite{LLV} covers a more general case, determining a more complex random source;  the main issue there was to determine which Gibbs probability $\mu_{\theta_0}$ (associated to a certain parameter $\theta_0\in \Theta$) is responsible for the generation of the samples obtained from the random source. There it was used a LDP version for the non additive case.

 In Section \ref{sec:2} we introduce for each value $n$ the corresponding
Neyman-Pearson test and we describe some basic properties. Section \ref{sec:3} considers asymptotic results, when $n \to \infty$, and large deviation estimates. In Section \ref{sec:4} we consider  minimax tests and Bayes solutions. In Section \ref{sec:5} we present some classical results on large deviations for  thermodynamic formalism (see  \cite{De}, \cite{Orey}, \cite{Ki}, \cite{L4} and \cite{Elis} for general references).

 For results somehow related to Statistics on a dynamical setting we refer the reader to  \cite{Denk}, \cite{Nobel1}, \cite{Nobel}, \cite{MN}, \cite{MMP}, \cite{Ji}, \cite{CGL} and \cite{CRL}. 
 Classical results in Decision Theory   can be found in \cite{Fer}, \cite{Rohatgi}, \cite{Buck} or \cite{Abra}. Nice references in thermodynamic formalism are \cite{PP}, \cite{Bala} and \cite{Bow}.

\section{Neyman-Pearson Tests}\label{sec:2}

We keep the notation introduced in Section \ref{sec:1}, in particular the notation for the Markov partition $\gamma$. For $n\ge 0$, we denote by $\gamma_n$ ($n\in \mathbb Z_+=\{0,1,2,...\}$) the refinement of the partitions $T^{-j}\gamma$, with $j=0,...,n$. We also use the notation  $S_ng= g+g\circ T+...+g\circ T^n$ for a measurable map $g:\Omega\to \mathbb R$ and $T^{-n-1}_\Gamma$ for the inverse mapping of $T^{n+1}:\Gamma\to T^{n+1}\Gamma$, where $\Gamma\in \gamma_n$.
Finally, $\mathbb I_B$ stands for the characteristic function of the set $B$.

By the eigenvalue property (see \cite{PP}) of a Gibbs measure we have for $\Gamma\in \gamma_n$, $i=0,1$:
\begin{equation} \label{imp1} \mu_i(\Gamma)= \int \mathcal L_{\log J_i}^n \mathbb I_\Gamma  d\mu_i=\int_{T^{n+1}\Gamma} \exp\{S_{n}\log J_i(T_\Gamma^{-n-1}(z))\}\mu_i(dz).
\end{equation} 

\begin{lemma}\label{lem:2.1}[see e.g. \cite{Fer}, p. 201]
The Neyman-Pearson tests at level $\alpha$ are given by the formulas
\begin{equation}\label{eq:2.1} \phi_\alpha^*(\omega)=\begin{cases}1,\qquad& \mbox{if}\ \omega\in S\\
\alpha,   & \mbox{if}\ \omega\not\in S
\end{cases}
\end{equation}
for the test problem $\mathcal E$ and - for the test problem $\mathcal E_n$ ($n\ge 0$) - by
\begin{equation}\label{eq:2.2}
 \phi_{n,\alpha}^*(\omega) =\begin{cases}  1 \qquad & \omega\in \Gamma\in \gamma_n;\ \ \int_{T^{n+1}\Gamma} \exp\{S_n\log J_1(T^{-n-1}_{\Gamma} (z))\} \mu_1(dz)\\
 &\qquad > c_{n,\alpha} \int_{T^{n+1}\Gamma} \exp\{S_n \log J_0(T^{-n-1}_{\Gamma} ( z))\} \mu_0(dz)\\
 0 & \omega\in \Gamma\in \gamma_n;\ \ \int_{T^{n+1}\Gamma} \exp\{S_n\log J_1(T^{-n-1}_{\Gamma} ( z))\} \mu_1(dz)\\
 &\qquad < c _{n,\alpha}\int_{T^{n+1}\Gamma} \exp\{S_n\log J_0(T^{-n-1}_{\Gamma} (z))\} \mu_0(dz)\\
\chi_{n,\alpha}  & \omega\in \Gamma\in \gamma_n;\ \ \int_{T^{n+1}\Gamma} \exp\{S_n\log J_1(T^{-n-1}_{\Gamma} (z))\} \mu_1(dz)\\
&\qquad = c_{n,\alpha} \int_{T^{n+1}\Gamma} \exp\{S_n\log J_0(T^{-n-1}_{\Gamma} ( z))\} \mu_0(dz),
\end{cases} 
\end{equation}
where $c_{n,\alpha}\in \mathbb R_+$ and $\chi_{n,\alpha}\in [0,1]$ are uniquely determined  constants  so that
$$ \int \phi_{n,\alpha}^* d\mu_0^n=\alpha.$$
\end{lemma}
\begin{proof} Taking the sum of the two measures involved for each of the test problems as their dominating measure and computing the densities we find
for the test problem $\mathcal E$ the densities 
$$  \frac {d\mu_0}{d(\mu_0+\mu_1)}=\mathbb I_{\Omega\setminus S},\qquad \frac {d\mu_1}{d(\mu_0+\mu_1)}=\mathbb I_S$$
and for the test problem $\mathcal E_n$ $(n\ge 0)$
$$  \frac {d\mu_0^n}{d(\mu_0^n+\mu_1^n)}(\omega)= \mu_0(\Gamma),\qquad     \frac {d\mu_1^n}{d(\mu_0^n+\mu_1^n)}(\omega)=\mu_1(\Gamma),\qquad  
\omega\in \Gamma\in \gamma_n,$$

 Note that  for $\Gamma\in \gamma_n$  the value
$\mu_i(\Gamma)$, $i=0,1$, can be calculated by \eqref{imp1}.     The lemma follows from   the Neyman-Pearson lemma, as formulated in \cite{Fer}, page 201, for example, which says that the Neyman-Pearson tests are defined by the quotients $\omega\mapsto \mu_1(\Gamma)/\mu_0(\Gamma)$, where $\omega\in \Gamma\in \gamma_n$.
\end{proof}

 {\it Remark:}  It follows from properties of the relative entropy of $\mu_0$ and $\mu_1$ (which are two distinct ergodic probabilities), that when $n$ goes to infinity, the quotients of the integrals in each line of \eqref{eq:2.2}  will go to zero or infinity (see for instance \cite{CRL}). The value $c_{n,\alpha}$ in some sense  calibrate numerically these quotients.  Therefore, the values $0$ or $1$, in the test  defined by (2.3), will discriminate, when $n$ is large, if the samples are being produced by the randomness of  $\mu_0$ or $\mu_1$.

 It follows immediately from the Neyman-Pearson lemma  that these tests are optimal in the sense that the type 2 error $\int (1-\phi)d\mu_1$ is minimal among all tests at level $\le \alpha$. This is 
 
 \begin{corollary}\label{cor:2.2}
 The Neyman-Pearson tests defined in (\ref{eq:2.1}) and (\ref{eq:2.2}) are most powerful tests at their respective  significance  levels $\alpha$. 
 \end{corollary}

 Let $L: \mathcal P\times\{0,1\}\to \mathbb R_+=\{z\in \mathbb R|\ z\ge 0\}$ be a loss function and denote
 $$\mathcal R(\mu,\phi) =\int L(\mu, t)  \delta_\phi(\omega, dt))  \mu(d\omega)$$ 
 the associated risk function, where $\delta_\phi$ denotes the decision function associated to the test $\phi$, that is
 $$ \delta_{\phi}(\omega,\cdot)= \phi(\omega)\mathbb  I_{\{0\}}+ (1-\phi(\omega))\mathbb I_{\{1\}}.$$
  In the sequel we consider w.l.o.g. the Neyman-Pearson loss function for the simple test problem, that is
 $$ L(\mu,t)=\begin{cases} 1 \qquad &\mbox{if} \ \mu\in \mathcal H_0, t=1\ \mbox{or}\ \mu\not\in \mathcal H_0, t=0\\
 0& \mbox{else}.
 \end{cases} $$
 
 Recall that a test $\phi$ is called a minimax test if 
 $$ \mathcal R(\phi):=\sup_{i\in \{0,1\}} \mathcal R(\mu_i,\phi)\le \inf_{\phi'}\sup_{i\in \{0.1\}} \mathcal R(\mu_i,\phi')=:\mathcal R(\phi')$$
 holds where $\mathcal R(\cdot, \phi')$  denotes the risk function of an arbitrary decision $\phi'$. $\mathcal R(\phi)$  will be called the risk of the test (decision) $\phi$.
 
 Likewise a test $\phi$ is called a Bayes solution for  the a priori distribution $\pi=(\pi_0,\pi_1)$ if 
 $$\mathcal R_\pi(\phi):= \int \mathcal R(\mu_i,\phi)\pi(di)\le \int \mathcal R(\mu_i,\phi') \pi(di)=:\mathcal R_\pi(\phi')$$
 holds for any test $\phi'$. The Bayes risk of the test $\phi$ with respect to the a priori distribution $\pi$ is $\mathcal R_\pi(\phi)$.
 
 A well known consequence of Corollary \ref{cor:2.2} is
 
  \begin{proposition}\label{prop:2.3}
 Let $\phi$ be a minimax test (or a Bayes solution with respect to the a priori distribution $\pi$,  $\mu_i^n$ and  $\mathcal E_n$). Then there exists a Neyman-Pearson test with the same risk function.
 \end{proposition}
 
 \begin{proof} Fix $n\in \mathbb N$. By definition 
 $$ \sup_{i\in \{0,1\}} \mathcal R(\mu_i^n,\phi)  \le \sup_{i\in \{0,1\}} \mathcal R(\mu_i^n,\phi')$$
 for all tests $\phi'$ of the test problem $\mathcal E_n$. Let $\alpha= E_{\mu_0^n}(\phi)$ denote the level of the test $\phi$. Then $\phi^*_{n,\alpha}$  has level $\alpha$ as well and $E_{\mu_1^n}(1-\phi_{n,\alpha}^*) \le E_{\mu_1^n}(1-\phi)$ so that
 $$ \mathcal R(\mu_\theta^n,\phi_{n,\alpha}^*)\le \mathcal R(\mu_\theta^n,\phi) ,\qquad \theta\in\{0,1\}.$$ 
 A similar argument works for the Bayes solution.
 
 This implies the next proposition.
 \end{proof}
 \begin{proposition}\label{prop:2.4}
 For each $n\ge 0$, there exists a minimax test and a Bayes solution to every a priori distribution $\pi$.
 \end{proposition}
 \begin{proof} The function
 $$  [0,1]\ni \alpha\mapsto \int \phi_{n\,alpha}^* d\mu_i^n$$
 is continuous for each $i=0$ and $i=1$.  Indeed, if $\alpha$ increases also the corresponding $c_{n,\alpha}$ decreases, and if $c_{n,alpha}$ is constant on some interval $(\alpha_0,\alpha_1)$ the corresponding $\chi_\alpha$ is increasing. Thus 
 \begin{eqnarray*}
 && \int 1-\phi_{n\alpha}^* g\mu_1=\\
 && = \mu_1(\frac{d\mu_1}{d\mu_0+\mu_1} < c_{n,\alpha} \frac{d\mu_0}{d\mu_0+\mu_1})+ (1-\chi_{n,\alpha})\mu_1(\frac{d\mu_1}{d\mu_0+\mu_1} \le c_{n,\alpha} \frac{d\mu_0}{d\mu_0+\mu_1})
 \end{eqnarray*}
 is decreasing and depends continuously on $\alpha$.
 Therefore, the minimum of  $\pi_0\int \phi_{n,\alpha
 }^* d\mu_0 +\pi_1\int (1-\phi_{n,\alpha}^*) d\mu_1$ is attained, so it is a Bayes solution.
 
 A similar argument works for the minimax test.
  \end{proof}
  
  \section{Large deviation and Neyman-Pearson tests}\label{sec:3}

  We keep the notation from the last sections. Let $\mathcal E$ and $\mathcal E_n$ denote the test problems described in Section \ref{sec:1}. We denote by  $J_i$, $i=0,1$, the Jacobians corresponding, respectively, to $\mu_i$, $i=0,1$ (cf. (\ref{eq:1.2})). Accordingly,  (\ref{eq:1.1}) will be taken with respect to these Jacobians. Furthermore,  for each $n\ge 1$ and $0\le \alpha\le 1$ the Neyman-Pearson test for the test problem $\mathcal E_n$ at level $\alpha$ is denoted by $\phi^*_{n,\alpha}$, see Lemma \ref{lem:2.1}.
  
  We shall use several facts from large deviation theory for Gibbs measures which are collected in an appendix (Section \ref{sec:5}).
  
  \begin{theorem}\label{theo:3.1}
  The free energy functions
  $$ f_i: \mathbb R\to \mathbb R,\quad (i\in \{0,1\})$$
 $$ f_i(t) = \lim_{n\to\infty} \frac 1n \log \int \exp\left\{t \log \frac {\int_{T^n\Gamma} \exp\{S_n\log J_{i'}(T_\Gamma^{-n} (z))\} \mu_{i'}(dz)}{\int_{T^n\Gamma} \exp\{S_n\log J_i(T_\Gamma^{-n} (z))\} \mu_{i}(dz)}\right\} d\mu_i,$$
 where $i'=i+1\mod 2$,
 exist, are twice differentiable and satisfy 
 \begin{eqnarray*}
  f_i(t)&=& P(\log J_i+t(\log J_{i'}-\log J_i))\\
  f_i'(t)&=& \int (\log J_{i'}-\log J_i) dm_{i,t}\\
  f_i''(t)&=& \lim_{n\to\infty} \frac 1n \int [S_n(\log J_{i'}-\log J_i- f_i'(t))]^2 dm_{i,t},
      \end{eqnarray*}
where for $i=0,1$ $m_{i,t}$ denotes the unique Gibbs measure for the potential 
$\log J_{i,t}=\log J_i+t(\log J_{i'}-\log J_{i})$  and where $P(\cdot )$ denotes the pressure function (its definition is recalled in the appendix).
  \end{theorem}
   \begin{proof} Let $i\in \{0,1\}$ and let $t$ be fixed.
 There exists a constant $K$ such that for $\Gamma\in \gamma_n$, $n\ge 1$ and $i\in \{0,1\}$
 \begin{eqnarray}\label{eq:3.1}
 && K^{-1}\le \frac {\mu_i(\Gamma)}{\exp\{ -n P(\log J_i) + S_n\log J_i(z)\}}\le K,\qquad z\in \Gamma,\notag \\
 && K^{-1}\le \frac {m_{i,t}(\Gamma)}{\exp\{ -n P(\log J_{i,t}) + S_n\log J_{i,t}(z)\}}\le K,\qquad z\in \Gamma,\\
 && K^{-1} \le \exp\{ S_n\log J_i(z)-S_n\log J_i(y)\}\le K,\qquad z,y\in \Gamma.\notag
 \end{eqnarray} 
 
 Writing 
 \begin{equation}\label{eq:3.2}
 G_n^i(\omega)= \int_{T^n\Gamma} \exp\{S_n\log J_i(T_\Gamma^{-n} z)\} \mu_i(dz),
 \end{equation}
 for $\omega \in \Gamma$  and using (\ref{eq:3.1}) it follows  that
 \begin{equation*}
 \left|\log G_n^{i'}(\omega)-\log G_n^i(\omega)-S_n\log \frac {J_{i'}}{J_i}(\omega)\right|  \le  2\log [K] +\log \frac {\mu_{i'}(T^{n+1}\Gamma)}{\mu_i(T^{n+1}\Gamma)} 
 \end{equation*}
 and so
 \begin{eqnarray}\label{eq:3.3}
&&  \lim_{n\to\infty} \frac 1n \log \int \exp\{t(\log G_n^{i'}(\omega)-\log G_n^i(\omega))\} \mu_i(d\omega) \notag \\
&& = \lim_{n\to\infty}\frac 1n \log \int \exp  (S_n (t\log \frac{J_{i'}}{J_i} ))(\omega) \mu_i(\omega). 
\end{eqnarray}
 Considering \eqref {eq:5.6} and \eqref {eq:5.7}, now apply relation (\ref{eq:5.3}) in the Appendix \ref{sec:5} to conclude that
\begin{eqnarray*}
&&  \lim_{n\to\infty} \frac 1n \log \int \exp\{t(\log G_n^{i'}(\omega)-\log G_n^i(\omega))\} \mu_i(d\omega) \\
&& = P(\log J_i+t(\log J_{i'}-\log J_i))-P(\log J_i).
 \end{eqnarray*}

$P(\log J_i)=0$, we arrive at
 \begin{eqnarray*}
  f_0(t) &=& P(\log J_0+t(\log J_1-\log J_0))\\
  f_1(t) &=& P(\log J_1+t(\log J_0-\log J_1)).
  \end{eqnarray*}
  The differentiability properties are well known for the pressure function (see  equations \ref{eq:5.2}, \ref{eq:5.4} and \ref{eq:5.5} in the Appendix).
 \end{proof}
 
 It is known that the ranges of the derivatives, restricted to $\mathbb R_+$, are $[f_i'(0), A_i]$
 with
 $$ A_i = \lim_{n\to \infty} \mbox{ess sup}\frac 1n  S_n(\log J_{i'}-\log J_i),$$
 the essential supremum is taken with respect to $\mu_i$, where  $i=\{0,1\}$, $i'=i+1 \mod 2$. Likewise the lower bounds for the ranges of the $f_i'$'s on $\mathbb R$ are
 \begin{equation*}
 \overline{A}_i= \lim_{n\to\infty} \mbox{ess inf} \frac 1n S_n(\log J_{i'}-\log J_i).
 \end{equation*}
 Since both measures are supposed to be strictly positive on all cylinders, we have $A_i=-\overline{A}_{i'}$, $i=0,1$,  $i'=i+1\mod 2$.
  
 \begin{lemma}\label{lem:3.2} For $i\in \{0,1\}$ and $i'=i+1\mod 2$, we have
 $$ \lim_{n\to\infty} \frac 1n \log \frac {G_n^{i'}}{G_n^i}= f_i'(0)\qquad \mbox{$\mu_i$ a.s.}$$
 \end{lemma}

\begin{proof} Note that
$$\frac 1n \log \frac {G_n^{i'}}{G_n^i} = \frac 1n\left(S_n(\log J_{i'}-\log J_i)\right)+ o(n)$$
by the proof of Theorem \ref{theo:3.1}, equation (\ref{eq:3.3}). Moreover,
$$ f_i'(0)= \int (\log J_{i'}-\log J_i)\, d\mu_i$$
since $\mu_i$ is the equilibrium measure for the potential $\log J_i +t(\log J_{i'}-\log J_i)$ when $t=0$. This proves the lemma by the ergodic theorem.
 \end{proof}
 The large deviation property of the Neyman-Pearson tests can now be formulated in 
  \begin{theorem}\label{theo:3.3} For any $n\ge 1$, let  $c_n\in \mathbb R_+$ be so that $c=\lim_{n\to\infty} \frac 1n \log c_n$ exists and let $\phi^*_{n,\alpha_n}$ denote a sequence of Neyman-Pearson tests for the test problem $\mathcal E_n$ with constants $c_n$ given in (\ref{eq:2.2}).
  \begin{enumerate}
  \item The type $1$ errors satisfy:\newline 
   If $c= f_0'(t)\in (f_0'(0), A_0]$ then
  \begin{equation}\label{eq:3.4}
  \lim_{n\to\infty} \frac 1n \log \int \phi^*_{n,\alpha_n} d\mu_0
= -t f_0'(t) + f_0(t).
 \end{equation}
  If $c>A_0$ then
\begin{equation}\label{eq:3.5}
  \lim_{n\to\infty} \frac 1n \log \int \phi^*_{n,\alpha_n} d\mu_0
= -\infty.
 \end{equation}
  If $c<f_0'(0)$ then 
 \begin{equation}\label{eq:3.6}
  \lim_{n\to\infty} \frac 1n \log \int \phi^*_{n,\alpha_n} d\mu_0
= 0.
 \end{equation}
 .
 \item The type $2$  errors satisfy
 \newline  If $ c= -f_1'(t) \in [-A_1, -f_1'(0))$, then
 \begin{equation}\label{eq:3.7}
\lim_{n\to\infty}\frac 1n \log \int (1-\phi^*_{n,\alpha_n})d\mu_1 =   -t  f_1'(t)+ f_1(t).
  \end{equation}
  If $c< -A_1$, then
\begin{equation}\label{eq:3.8}
  \lim_{n\to\infty} \frac 1n \log \int \phi^*_{n,\alpha_n} d\mu_0
= -\infty.
 \end{equation}
  If $c\geq -f_1'(0)$, then
 \begin{equation}\label{eq:3.9}
  \lim_{n\to\infty} \frac 1n \log \int \phi^*_{n,\alpha_n} d\mu_0
= 0.
 \end{equation}
  \end{enumerate}
  \end{theorem}
  
  \begin{proof} \begin{enumerate}
 \item We show the first case (\ref{eq:3.4}). Using the notation in (\ref{eq:3.2}),  for a suitable $\chi\in [0,1]$, chosen according to (\ref{eq:2.2}),
 \begin{eqnarray*}
&& \mu_0\left( \frac{G_n^1}{G_n^0} > c_n\right) 
\le \int \phi^*_{n,\alpha_n} d\mu_0\\
&& =\mu_0\left( \frac{G_n^1}{G_n^0}> c_n\right) + \chi \mu_0\left( \frac{G_n^1}{G_n^0}=c_n\right)\\
&& \le  \mu_0\left( \frac{G_n^1}{G_n^0}\ge c_n\right)  
  \end{eqnarray*}
  
 By Markov's inequality for  all $t >0$
 \begin{eqnarray*}
 && \int \phi^*_{n,\alpha_n} d\mu_0\le 
 \mu_0\left( t\log \frac {G_n^1}{G_n^0}\ge t\log c_n\right)\\
 &&\quad\le 
 \exp\{-t\log c_n+ n f_{0,n}(t)\},
 \end{eqnarray*}
 where
 $$ f_{0,n}(t)= \frac 1n \log \int  \exp\{t\log \frac{G_n^1}{G_n^0}\}d\mu_0.$$

 Taking the infimum over $t>0$ yields  for $n$ sufficiently large
 $$ \int \phi^*_{n,\alpha_n} d\mu_0\le  K_1 \exp\{ -tn f_0'(t)+n f_0(t)\},$$
 where $t$ satisfies $ f_0'(t)=\lim_{m\to\infty} \frac 1m   \log c_m$ and $K_1$ is some universal constant.

  For the lower bound of (\ref{eq:3.4}) note that a  Gibbs measure $m$ with Jacobian $J$ satisfies (see (\ref{eq:3.1}) and by $T$-invariance)
$$K^{-3}\le  \frac{m([c_0,...,c_{p+q-1}])}{m([c_0,...,c_{p-1}])\cdot m([c_{p},...,c_{p+q-1}])}\le K^3$$
for $p,q \ge 1$ and $[c_0,...,c_{p+q-1}]\ne\emptyset$.
Moreover, for a topologically mixing subshift of finite type there exists a constant $r\ge 1$ such that any cylinders $c,d\subset \Omega$ the set $c\cap T^{-r} d\ne \emptyset$.
  Since  for $\omega\in \Gamma\in \gamma_n$ by (\ref{eq:3.1})
 $$ \log \frac{G_n^1}{G_n^0}(\omega)-nf_0'(0)\ge K^2\cdot \frac{\mu_1(T^n\Gamma)}{\mu_0(T^n\gamma)} \exp \{S_{n+1}[\log J_1J_0^{-1}- f_0'(0)](\omega)\},
 $$
it follows that for $n$ sufficiently large 
  \begin{eqnarray*}
 && \int \phi^*_{n,\alpha_n} d\mu_0\ge  \mu_0\left( \log \frac{G_n^1}{G_n^0}-nf_0'(0)> \log c_n-nf_0'(0)\right) \\
 &&\quad \ge \mu_0( S_{n+1}(\log J_1J_0^{-1}- f_0'(0))\ge \log c_n- nf_0'(0)+O(1) .
  \end{eqnarray*}
  Therefore the proof of Theorem 3.3 in \cite{DK} applies with minor adaptions as well for this case, proving the lower bound.  In order to see this, note that the coordinate process of a Gibbs measure is $\psi$-mixing, so Theorem 3.3 in \cite{DK} is applicable here to partial sums above in view of (\ref{eq:3.1}). Alternatively, the arguments for its proof also work for cylinders. 
 Moreover, one also can use \cite{Orey}.
  
 Now we will show (\ref{eq:3.5}) and (\ref{eq:3.6}). 
  If $c> A_0$ and $t>0$, we have $\frac d{dt} [-tnc +n f_0(t)]= -nc +n f_0'(t) \le C<0$, for some $C<0$, so that the infimum is attained for $t\to \infty$.
   
   If $c<f_0'(0)$ then by Lemma \ref{lem:3.2} $\log \frac{G_n^1}{G_n^0}-f_n'(0)\to 0> c-f_0'(0)$, $\mu_0$ a.s..
   
 \item  Using the notation in (\ref{eq:3.1}),  for a suitable $\chi\in [0,1]$
 \begin{eqnarray*}
&& \mu_1\left( \frac{G_n^1}{G_n^0} < c_n\right) 
\le \int 1-\phi^*_{n,\alpha_n} d\mu_1\\
=&& \mu_1\left( \frac{G_n^1}{G_n^0}< c_n\right) + (1-\chi) \mu_1\left( \frac{G_n^1}{G_n^0}=c_n\right)
\le  \mu_1\left( \frac{G_n^1}{G_n^0}\le c_n\right)\\
=&& \mu_1\left( \frac{G_n^0}{G_n^1}\ge \frac 1{c_n}\right)
  \end{eqnarray*}
Now this case is handled as  case $i=0$.
\end{enumerate}
  \end{proof}

 \section{Minimax tests and Bayes solutions}\label{sec:4} 
 
 Here we prove the rate of convergence for the risk of the minimax tests and Bayes solutions  in $\mathcal E_n$. We discuss the case of minimax tests first, the analogous arguments work for the Bayes solutions so that we only formulate those results.
 
 We begin with
 \begin{lemma} \label{lem:4.1} There exists a minimax test $\psi_n^*$  in $\mathcal E_n$ that satisfies
 \begin{equation}\label{eq:4.1}
  \int \psi_n^* d\mu_0= 1-\int \psi_n^* d\mu_1.
  \end{equation}
  In particular, this test can be chosen to be a Neyman-Pearson test.
 \end{lemma}
 
 \begin{proof} Let $\alpha$ denote the significance level of a minimax test $\psi_n^*\in \mathcal T_n$, where $n$ is some fixed integer. Let 
 $\beta= \int \psi_n^* d\mu_1$ be its power. 
  
 If $1-\beta>\alpha$ then a Neyman-Pearson test $\phi_{n,\alpha'}^*$ at level $\alpha'\in [\alpha, 1-\beta]$ has at most a type 2 error of $1-\beta$, because it has a lower type 2 error than $\psi_n^*$. 
 If for all such $\alpha'$ the Neyman-Pearson test has power $\beta$, then $\phi_{n, 1-\beta}^*$ is a Neyman-Pearson test satisfying the requirements of the lemma. If the power is strictly larger than $\beta$ for some $\alpha'$, then the test $\phi_{n,\alpha'}^*$ has a smaller risk than $\psi_n^*$, which is impossible.
 This proves the lemma if $1-\beta>\alpha$.
 
 If $1-\beta=\alpha$ the assertion follows from the same argument as has been used in the proof of Proposition \ref{prop:2.3}.

 If $1-\beta < \alpha$,  a Neyman-Pearson test at level $\alpha$ has a power larger than or equal to $\beta$. This implies
 $$ \mathcal R(\phi_{n,\alpha}^*)= \alpha= \mathcal R(\psi_n^*)$$
 and
 $$ \int \phi_{n,\alpha}^* d\mu_1 \ge \beta > 1-\alpha.$$
 Assume that
 $$ \int \phi_{n,\alpha}^* d\mu_0 =\alpha > 1-\int \phi_{n,\alpha}^* d\mu_1.$$
Since the power of a Neyman-Pearson test is continuous, there is 
 $\alpha'<\alpha$ such that 
 $$\int  \phi_{n,\alpha'} d\mu_1 > \beta -( \beta-(1-\alpha))= 1-\alpha,$$
 hence 
 $$ \max\{\alpha', 1-\int \phi_{n,\alpha'}^* d\mu_1\}= \mathcal R(\phi_{n,\alpha'}^*)<\alpha= \mathcal R(\psi_n^*),$$
 a contradiction. 
 
 This finishes the proof.
 \end{proof}
 
 \begin{lemma}\label{lem:4.2} For $i=0,1$ let $i'=i+1\mod 2$.
Let 
$$F_i(t) =tf_i'(t)-f_i(t)=t \int (\log J_{i'}-\log J_i) dm_{i,t}-  P(\log J_i+t(\log J_{i'}-\log J_i)),$$
 $i=0,1$ and $i'=i+1\mod 2$,
 denote the (information) functions in (\ref{eq:3.4}) and (\ref{eq:3.7}), where $m_{i,t}$ denotes the unique  equilibrium measure for the potential $\log J_i+t(\log J_{i'}-\log J_i)$. Then
\begin{enumerate}
\item The functions $F_i$, $i=0,\,1$,  are increasing on $(0,\infty)$.
\item $F_i(0)=0$ for $i=0,\, 1$.
\item For $0\le t\le 1$ one has $m_{i,t}=m_{i',-t+1}$, $m_{i,0}=\mu_i$, $m_{1,1}=\mu_0$ and $m_{0,1}=\mu_1$.
\item For $0\le t\le 1$ and $i=0,1$ one has $F_i(t) = F_{i'}(-t+1)- 2t\int (\log J_{i'}-\log J_i )\,dm_{{i'},-t+1}$.
\item  For $i=0,1$ one has $F_i(1)= \int (\log J_i-\log J_{i'}) d\mu_{i'}\geq 0$.
\end{enumerate}
\end{lemma}
 \begin{proof} \begin{enumerate}
 \item The  derivative of $F_i$ equals $F_i'(t)= tf_i''(t)$ which is positive on $\mathbb R_+$.
 \item
 $$ F_i(0) = -f_i(0)= -P(\log J_i)=0.$$
 \item  Note that  $P(\log J_i+t(\log J_{i'}-\log J_i))= P(\log J_{i'} +(t-1)(\log J_{i'}-\log J_i))$ so that $m_{i,t}=m_{i',-t+1}$. 
 \item  By 3.~it follows that
 \begin{eqnarray*}
  F_i(t) &=& t\int (\log J_{i'}-\log J_{i})\, dm_{i,t} - P(\log J_i+t(\log J_{i'}-\log J_i))\\
  &=& - t\int (\log J_i-\log J_{i'})\, dm_{i',-t+1} - P(\log J_{i'} +(-t+1)(\log J_{i}-\log J_{i'}))\\
  &=& F_{i'}(-t+1)- 2t\int (\log J_{i}-\log J_{i'}) dm_{i',-t+1}.
  \end{eqnarray*}
  \item This is obvious from 2., 3. and 4.: $F_i(1)= \int (\log J_{i'}-\log J_{i})\, dm_{i,1}= \int (\log J_{i'}-\log J_{i}) \ d\mu_{i'}$.
 It follows from the variational principle and Rohklin's formula  that
 $$ \int (\log J_{i'}-\log J_{i}) \, d\mu_{i'}= -[h_{\mu_{i'}}(T) +\int \log J_i \ d\mu_{i'}]\ge -P(\log J_i)=0.$$
 \end{enumerate}
 \end{proof}
  
  \begin{theorem}\label{theo:4.3} Let $\psi_n^*$ be a sequence of minimax tests in $\mathcal E_n$, $n\ge 1$. Then their risks 
  $$ \mathcal R(\psi_n^*)= \max \{\int \psi_n^* d\mu_0,\ 1-\int \psi_n^* d\mu_1\}$$ 
  satisfy
  \begin{equation}\label{eq:4.2}
 \lim_{n\to\infty} \frac 1n \log \mathcal R(\psi^*_ n)
 \le \inf\{\max\{f_0(t)-tf_0'(t), f_1(s)-s f_1'(s)\}\}
 \end{equation}
 where the infimum extends over all pairs $(t,s)$ with
 $t\in (f_0'(0),A_0)$, $s\in (-A_0, -f_1'(0))$ and  $f_0'(t)=- f_1'(s)$.
  More precisely,
 \begin{equation}\label{eq:4.3}
 \lim_{n\to\infty} \frac 1n \log \mathcal R(\psi^*_n) = f_0(t_0)-t_0f_0'(t_0),
  \end{equation}
  where $t_0$ is the solution of the equations
  \begin{eqnarray}\label{eq:4.4}
  &&  f_0'(t)=- f_1'(s(t)) \\
  &&  F_0(t_0)=\min\{ F_0(t)| \  s(t)s'(t)f_1''(s(t))-tf_0''(t)=0\}.
  \notag \end{eqnarray}
  This solution is unique.
  \end{theorem}

  \begin{proof} By Lemma \ref{lem:4.1} we may assume that the tests are Neyman-Pearson tests satisfying (\ref{eq:4.1}).   Let  $\phi_{n,\alpha_n}^*:=\phi_{n}^* $  denote the Neyman-Pearson test for $\mathcal E_n$ with $\alpha_n= 1-\int \phi_{n,\alpha_n}^* d\mu_1$. Let $c_n$ denote the constant given by its definition as a Neyman-Pearson test.
  
  We first show that 
  $$ \overline{A}_0\le f_0'(0)\le  \liminf_{n\to\infty} \frac 1n \log c_n\le \limsup_{n\to\infty} \frac 1n\log c_n \le A_0.$$
  
  1) If there is a subsequence $\frac 1{n_k} \log c_{n_k}> A_0$, then for all $k$ sufficiently large
  $$ \int \phi_{n_k}^* d\mu_0 =0,$$
  since $\mu_0$-almost surely
  $$ \log {G_{n_k}^1}- \log G_{n_k}^0 \le  \log K+ n_k A_0< \log c_{n_k}.$$
  Likewise, 
  $$ \int ( 1-\phi_{n_k}^* ) d\mu_1 =1,$$
  since $-A_0=\overline{A}_1$ and 
   $$ \log {G_{n_k}^0}- \log G_{n_k}^1 \ge  -\log K- n_k  A_0 > -  \log c_{n_k}.$$
   
 2.   If there is a subsequence $\frac 1{n_k} \log c_{n_k}< f_0'(0)$, then for all $k$ sufficiently large
  $$ \int \phi_{n_k}^* d\mu_0 =1,$$
  since $\mu_0$-almost surely by Lemma \ref{lem:3.2}
  $$ \log {G_{n_k}^1}- \log G_{n_k}^0 \ge  \log K+ n_k f_0'(0) +o(n_k)> \log c_{n_k}.$$
  Likewise, 
  $$ \int (1-\phi_{n_k}^*) d\mu_1 =0,$$
  since by the variational principle $f_0'(0)= \int \log J_0-\log J_1 d\mu_0= -[h_{\mu_0}+\int \log J_1 d\mu_0]<0.$ Then,
  $$ \log {G_{n_k}^0}- \log G_{n_k}^1 \ge  -\log K- \log c_{n_k}> -n_k f_0'(0) > 0,$$
  and 
  $\mu_1$-a.s. by Lemma \ref{lem:3.2}
 \begin{eqnarray*}
 &&\frac 1{n} (\log {G_{n_k}^0}- \log G_{n_k}^1) \to \int (\log J_0-\log J_1) d\mu_1\\
 &&= h_{\mu_1}+ \int \log J_0 d\mu_1 < P(\log J_0)=0.
 \end{eqnarray*}

   This is a contradiction.
   
   It follows that the sequence $c_n$ satisfies
   $$ -A_1\le f_0'(0)\le \liminf_{n\to\infty} \frac 1n\log c_n \le \limsup_{n\to\infty} \frac 1n \log c_n \le A_0,$$
   that is : $c$ is contained in the image of the function $f_0'$.
   This also implies that $-c$ is contained in the interval $[-A_0, -f_0'(0)]\subset [\overline{A}_1, -\overline{A}_0]= [\overline{A}_1, A_1]$, which is the image of $f_1'$.

   Assume first that $c=\lim_{n\to\infty} \frac 1n \log c_n\in [f_0'(0),A_0]$  exists. Then there exists $t\ge 0$ with $c=f_0'(t)$.
 Moreover, $-c\in [-A_0, -f_0'(0)]\subset [\overline{A}_1, A_1]$ means that there is $s$ with $f_1'(s)=-c$.
It then follows that by  Theorem \ref{theo:3.3}
  $$\lim_{n\to\infty} \frac 1n \log \mathcal R(\phi_n^*)\le \max\{ -tf_0'(t)+ f_0(t),\ -sf_1'(s)+f_1(s)\}.$$
   By Lemma \ref{lem:4.1}  we also must have that
   $$ -tf_0'(t)+ f_0(t) =-sf_1'(s)+f_1(s)$$
   and this value must be minimal. Since by Lemma \ref{lem:4.2}  each of these functions is strictly decreasing, but the function $t\to s(t)$ defined by $f_0'(t)=-f_1'(s(t))$ is strictly increasing, it follows that the function $t\mapsto -s(t)f_1'(s(t))+f_1(s(t))$  is increasing. This means  that there is a unique $t_0$ with
   $$ -t_0f_0'(t_0)+f_0(t_0)= - s(t_0) f_1(s(t_0))+ f_1(s(t_0))$$
   and 
   $$ f_0'(t_0)= f_1'(s(t_0)).$$
   
   In particular, we must have that $c=\lim_{n\to\infty}\frac 1n \log c_n$ exists because the functions $f_i'$ are strictly increasing.
   \end{proof}

  Bayes solutions can be handled much in the same way as the minimax test. Let $\pi=(\pi_0, \pi_1)$ be probability vector and let
  $$ \mathcal R_\pi(\phi) = \pi_0 \int \phi d\mu_0+\pi_1\int (1-\phi) d\mu_1$$
to  denote the Bayes risk  for the Bayes distribution $\pi$ of the test $\phi$ given  the test problem $\mathcal E_n$.
  
  \begin{theorem}\label{theo:4.4} Let $\pi=(\pi_0,\pi_1)$ be a Bayes prior distribution. Then, the Bayes solutions  $\psi_{\pi,n}^*$ with respect to $\pi$ for the test problem $\mathcal E_n$ have risks satisfying
  $$\lim_{n\to \infty} \frac 1n \log \mathcal R_{\pi}(\psi_{\pi,n}^*)\le
 \inf \{ \pi_0(-tf_0'(t)+f_0(t)) +\pi_1 (-sf_1'(s)+f_1(s))\},$$
  where the infimum extends over all pairs $(s,t)$ so that $f_0'(t)=-f_1'(s)$.
  
  More precisely, let $t_\pi$ be chosen so that
  $$ -\pi_0(t_\pi f_0'(t_\pi)+f_0(t_\pi))+\pi_1(s(t_\pi) f_0'(t_\pi)+f_1s((t_\pi)))  $$
  and
  $$ f_0'(t_\pi)=f_1'(s(t_\pi)).$$
  Then $t_\pi$ is uniquely determined and satisfies
    $$\lim_{n\to \infty} \frac 1n \log \mathcal R_{\pi}(\psi_{\pi,n}^*)=
 2 \pi_0(-t_\pi f_0'(t_\pi)+f_0(t_\pi)).$$
  \end{theorem}
  
  \section{Appendix on large deviation}\label{sec:5}

  Let $d\ge 2$ be an integer and $M=(m_{ij})_{1\le i,j\le d}$ be an integral matrix with entries on $\{0,1\}$.
  Gibbs states on mixing subshifts of finite type
  $$ \Omega= \{(\omega_n)_{n\ge 0}|\ 1\le \omega_n\le d;\  m_{\omega_n,\omega_{n+1}}=1\}$$
  were introduced by Bowen in \cite{Bow}. For a given H\"older continuous function $g:\Omega\to \mathbb R$ there exists a Gibbs measure $\mu_g$ such that
  \begin{equation}\label{eq:5.1} h_{\mu_g}(T)+\int g(\omega) \mu_g(d\omega)= \max\{ h_m(T)+\int g(\omega) m(d\omega)|\ m\circ T^{-1}= m,\ m(\Omega=1\},
  \end{equation}
 where $h_m(T)$ denotes the entropy of the invariant probability $m$, which by Rohklin's theorem satisfies $h_m(T)= -\int \log J dm$ where $J^{-1}$ is the Jacobian of $m$ (see the introduction for our use of the Jacobian). In particular, a Gibbs measure $m$ for the potential $\log J$ satisfies $P(\log J) := h_m(T)+\int \log J dm= 0$ (Bowen's formula).
  The right hand side of  equation (\ref{eq:5.1}) can be chosen as a definition of the pressure $P(g)$  for any continuous function $g$, hence $P:C(\Omega)\to \mathbb R$. It is well known that the function $P$ is Gateaux differentiable in the sense that for H\"older continuous functions $g,h\in C(\Omega)$
  \begin{equation}\label{eq:5.2}  
  \frac d{dt} P(g+th)= \int h(\omega) m_t(d\omega),
  \end{equation}
  where $m_t$ denotes the Gibbs measure for the function $g+th$.
  The free energy function $f_h$ for a H\"older continuous function $h$ with respect to the Gibbs measure $\mu_g$ exists and satisfies
  \begin{equation}\label{eq:5.3}
   f_h(t)=\lim_{n\to\infty} \frac 1n \log \int \exp\{t\cdot S_nh\}  d\mu_g= P(g+th)-P(g)
   \end{equation}
  and hence is differentiable on its domain with first and second derivative
  \begin{equation}\label{eq:5.4}
   f_h'(t)=\int h dm_t, 
   \end{equation}
$m_t$ the equilibrium state for the potential $g+th$,  and
  \begin{equation}\label{eq:5.5}
   f_h''(t)= \lim_{n\to\infty} \frac 1n \int (S_n(h- \int h dm_t))^2  dm_t.\end{equation}
 The domain of  $f_h$ is the real line, but the range of its derivative is a subinterval $(a,b)\subset \mathbb R$
 defined by
 \begin{eqnarray*}
 a&=& \lim_{n\to\infty}  \mbox{ess inf}  \frac 1n S_nh\\
  b &=&  \lim_{n\to\infty}  \mbox{ess sup}  \frac 1n S_nh.
  \end{eqnarray*}

  We denote by  $z \to I(z)$ the Legendre transform of the analytic function $t \to f_h(t)$.
Then, 
\begin{enumerate}
\item given an open  interval $(a,b)\subset \mathbb{R}$,
\begin{equation}\label{eq:5.6}
\lim_{n \to \infty} \frac{1}{n} \log \mu_g \left\{\, y\in \Omega \,|\,  \frac{1}{n} \sum_{i=0}^{n-1} h (T^i (y))\in (a,b)\,\right\} \geq  -\, \inf \{ I(z) \,|\, z \in (a,b)\}.
\end{equation}
\item given a closed  interval $[a,b]\subset \mathbb{R}$,
\begin{equation}\label{eq:5.7}
\lim_{n \to \infty} \frac{1}{n} \log \mu_g \left\{\,  y\in \Omega \,|\,  \frac{1}{n} \sum_{i=0}^{n-1} h (T^i (y))\in [a,b]\,\right\} \leq  -\, \inf \{ I(z) \,|\, z \in [a,b]\}.
\end{equation}
\end{enumerate}

For a proof see \cite{De}, \cite{Ki} or \cite{L4}.

\section{Appendix on Statistical terminology and definitions} \label{sec:6}

\vspace{1cm}

Here we collect basic facts and definition on statistical decision theory which are used in this note. It is included to make the paper self-contained for the readership in dynamical systems.

A {\it statistical experiment} is a triple $\mathcal E:=(\Omega, \mathcal F, \mathcal P)$  consisting of a measurable space $(\Omega, \mathcal F)$ together with a family $\mathcal P$ of probability measures on $(\Omega,\mathcal F)$. Here $\mathcal F$ denotes a $\sigma$-algebra on $\Omega$. The objective is to make a decision about the {\it true} probability in $\mathcal P$ once a point in $\Omega$ is observed. For example, $\Omega$ may be chosen to be $\mathbb R^n$ and $\mathcal P$ may be chosen to be all  Gaussian  distributions  which are the $n$-fold product measure of a one-dimensional normal distribution with expectation $\mu\in \mathbb R$ and variance $1$. The objective may be to find the true $\mu$.

Decisions are made with certain probabilities. Formally this is described by a  measurable space  $(D,\mathcal D)$ (where $\mathcal D$ denotes the $\sigma$-algebra on $D$). It is called the space of decisions and a  decision function is a stochastic kernel
$$ \delta:\Omega\times \mathcal D \to [0,1]$$
with the interpretation that a decision is in $d\in \mathcal D$ with probability $\delta(\omega,d)$ provided the observation is $\omega\in \Omega$. Such decisions (decision functions) in $[0,1]$ are called randomized decisions (decision functions) in contrary to the non-random case when the probability $\delta(\omega, d)$ is either $0$ or $1$. Let us denote the collection of all  decision functions $\delta$ for a fixed statistical experiment by $\Delta$.

It is common in statistics to value a decision  using loss functions 
$$L:\mathcal P\times D\to \mathbb R_+$$  
 which measures the {\it loss} of a decision $d\in D$ when $P$ is "true". It is assumed that for all $P\in \mathcal P$ the function $d\to L(P,d)$ is measurable. 
 
 A {\it statistical problem} is then defined by $(\mathcal E, \Delta, L)$ where the triple is explained above. A {\it test problem} is a special statistical problem, specified by a subset $\mathcal H_0\subset \mathcal P$, the hypothesis, the decision space $D=\{0,1\}$ and a loss function $L$. $L$ takes on the form
 $$ L(P,d)=\begin{cases} L_0\qquad&\mbox{if}\ P\in \mathcal H_0\ \mbox{and}\ d=1\\
 L_1 &\mbox{if}\  P\not\in \mathcal H_0\ \mbox{and}\  d=0\\
 0& \mbox{else}
 \end{cases} $$
The test problem is called {\it simple} if $\mathcal H_0$ and its complement consist of exactly one probability. This scenario is underlying the present article  and we mostly assume that the loss function is of Neyman-Pearson type, that is $L_0=L_1=1$.

We restrict the discussion to the special case of a simple test problem since it is the objective in this paper. A test is the function
$$ \varphi: \Omega\to [0,1]$$
defined as
$$ \varphi(\omega)= \delta(\omega,\{1\})$$
where $\delta\in \Delta$ is a decision  function. Randomized and non-randomized tests are those where $\delta$ has the corresponding property.  The decision $d=1$ then means that the observation suggests that the unknown distribution in $\mathcal P$ does not belong to $\mathcal H_0$ while $d=0$ means that the unkown distribution belongs to $\mathcal H_0$. In the first case one rejects the hypothesis while in the second one does not reject the hypothesis.

Finally, tests $\varphi$ for simple test problems are rated by their risk functions
\begin{eqnarray*}
&& R(\cdot, \delta): \mathcal P\to [-\infty,\infty]\qquad \varphi= \delta(\cdot,\{1\})\\
&&  R(P,\delta)= \int_\Omega\int_D  L(P, \delta(\omega,t) \delta(\omega,d t) P(d\omega),
\end{eqnarray*} 
which amounts to the type 1 error (or significance level)
$$ \alpha:= R(P, \delta)= L_0 \int \varphi(\omega) P(d\omega),$$
for $P\in \mathcal H_0$ and to the  type 2 error
 $$ 1-\beta:= R(P, \delta)= L_1 \int1- \varphi(\omega) P(d\omega),$$
for $P\not\in \mathcal H_0$. The value $\beta$ is called the {\it power} of the test. One may assume that $L_0=L_1$ when comparing tests.

The Neyman-Pearson Lemma characterizes those tests which have maximal power subject to keeping a given significance level $\alpha$. It reads

Let $\mathcal P=\{\mu_0,\mu_1\}$, $\mathcal H_0=\{\mu_0\}$ and $\mu$ be a dominating measure for $\mu_i$, $i=0,1$, for example $\mu=\mu_0+\mu_1$. Let $f_i$ denote the densities of $\mu_i$ with respect to $\mu$.

A Neyman-Pearson test $\varphi$ is  a test of the form
$$ \varphi(\omega)=\begin{cases} 1\qquad &\mbox{if}\ f_1(\omega)> C f_0(\omega)\\
0&\mbox{if}\ f_1(\omega)< C f_0(\omega)\\
\gamma(\omega) &\mbox{if}\ f_1(\omega)= C f_0(\omega)
\end{cases}$$
where $\gamma(\omega)\in[0,1]$ and $C\in [-\infty,\infty]$.
In particular one may choose $\gamma$ to be constant on $\{f_1=C f_0\}$.
Then we have the following facts:
\newline 1. A Neyman-Pearson test $\varphi$ has maximal power among all tests $\psi$ with
$$ \int \psi d\mu_0 \le \int \varphi d\mu_0.$$
\newline 2. Given $\alpha\in[0,1]$ there is a Neyman-Pearson test $\varphi$ satisfying
$$ \int \varphi d\mu_0=\alpha.$$
\newline 3. A test $\psi$ at significance level $\alpha$ and with maximal power among all tests with significance level $\alpha$ is  a.s. a Neyman-Pearson test.

\end{document}